\documentclass{amsart}
\usepackage[margin=1.2in]{geometry}

\usepackage{amssymb}
\usepackage{amsmath}
\allowdisplaybreaks

\usepackage{xcolor}
\usepackage[all]{xy}
\usepackage{hyperref}
\usepackage{amssymb}
\usepackage{amsthm}
\usepackage{tikz-cd}

\usepackage{diagbox}
\usepackage{multirow}
\usepackage{float}

\newcommand{\Q}{\mathbb{Q}}

\newcommand{\Z}{\mathbb{Z}}

\newcommand{\dd}{(d)}

\newcommand{\sm}{\,| \,}

\theoremstyle{plain}
\newtheorem{thm}{Theorem}
\newtheorem{lem}[thm]{Lemma}

\theoremstyle{remark} \newtheorem{remark}[thm]{Remark}

\begin{document}
\title{An orthogonal relation on inverse cyclotomic polynomials} 

\author{Jianfeng Xie}
\address{University of Science and Technology of China, Hefei, Anhui 230026, China}

\email{xjfnt@mail.ustc.edu.cn}


\maketitle

 \begin{abstract} Let $\Phi_n(X)$ and $\Psi_n(X)=\frac{X^{n}-1}{\Phi_{n}(X)}$ be the $n$-th cyclotomic and inverse cyclotomic polynomials  respectively. In this short note,  for any pair of divisors $ d_{1} \neq  d_{2} $ of $ n $, and  integers $l_1$ and $l_2$ such that  $ 0 \leq l_{1} \leq \varphi(d_{1})-1 $ and $ 0 \leq l_{2} \leq \varphi(d_{2})-1  $, we show that
	\[\left \langle X^{l_{1}} \Psi_{d_{1}}(X) (1+X^{d_1}+\dots X^{n-d_1}), X^{l_{2}} \Psi_{d_{2}}(X) (1+X^{d_2}+\dots X^{n-d_2}) \right \rangle  =0, \]
where $ \langle \cdot, \cdot \rangle $ is the inner product on $\Q[X]$ defined by $ \langle \sum_{k} a_{k}X^{k},\sum_{k} b_{k}X^{k}  \rangle =\sum_{k} a_{k}b_{k}$. 
\end{abstract}

For a positive integer $ n $, the $ n $-th cyclotomic polynomial $ \Phi_{n}(X) $ is known to be an irreducible polynomial in $\Z[X]$ of degree $\varphi(n)$. The $ n $-th inverse cyclotomic polynomial, introduced in \cite{Mo}, is 
\[\Psi_{n}(X):=\frac{X^{n}-1}{\Phi_{n}(X)}.   \]
The (inverse) cyclotomic polynomials are interesting research objects in elementary and algebraic number theory, see for example, the bibliography of \cite{Apo, HPM} for more details. Recently, research in this area is concentrated on  coefficients of (inverse) cyclotomic polynomials, especially when the number of prime divisors of $n$  is small (see \cite{Bei, GM, JLM}).

However, up to now, there seems still no results revealing the relations between  $ \Psi_{m}(X) $ and $ \Psi_{n}(X) $. In this note, we shall present an  orthogonal relation between $ \Psi_{m}(X) $ and $ \Psi_{n}(X) $. 

In the polynomial ring $\Q[X] $, we set
\[\langle\sum_{k=0}^{\infty}a_{k}X^{k}, \sum_{k=0}^{\infty}b_{k}X^{k}  \rangle:=\sum_{k=0}^{\infty}a_{k}b_{k}.   \]
For $ d  \mid n$, set
\[\Psi_{n,d}(X):=\frac{X^{n}-1}{\Phi_{d}(X)}=(1+X^{d}+\dots X^{n-d}) \Psi_d(X) .   \]

Our main result is:
\begin{thm}\label{coeorth}
For any pair of divisors $ d_{1}\neq d_{2} $ of $ n $, and any pair of integers $l_1$ and $l_2$ that $ 0 \leq l_{1} \leq \varphi(d_{1})-1 $ and $ 0 \leq l_{2} \leq \varphi(d_{2})-1 $, one has
	\[ \langle X^{l_{1}} \Psi_{n,d_{1}}(X), X^{l_{2}} \Psi_{n,d_{2}}(X)\rangle  =0.\]
\end{thm}
\begin{remark}	
Note that  the coefficients of $ \Psi_{n,d}(X) $ are exactly those of $ \Psi_{d}(X) $, with each repeating $ \frac{n}{d} $ times. Thus Theorem \ref{coeorth} can be regarded as certain orthogonal relation between coefficients of $ \Psi_{d_{1}}(X)$ and $ \Psi_{d_{2}}(X) $.
\end{remark}

We give a proof of this result based on  representation theory.

Let $ G $ be the cyclic group of order $ n $ and $g_0$  a fixed generator  of $ G $. We equip $\Q[X]/(X^n-1)$  with a $G$-action by $g_0\cdot f(X)=Xf(X)$. 
Then the map $g_0\mapsto X$ gives an isomorphism
\[\Q[G]\cong \Q[X]/(X^n-1) \]
of $ \Q[G] $-algebras. We identify these two via this isomorphism.  By the factorization $X^n-1=\prod_{d\mid n}\Phi_d(X)$, then
\[\Q[G]= \bigoplus\limits_{d \mid n}\Q[X]/(\Phi_{d}(X)) =\bigoplus_{d \sm n}\Q[G]^{\dd}, \] 
where  $ \Q[G]^{\dd}:=\Q[X]/\Phi_{d}(X)  $. By Maschke Theorem (see \cite[\S 3.6]{Pi}), $ \Q[G] $ is semisimple, and $ \{\Q[G]^{\dd}: d\mid n\}$ is the set of all (non-isomorphic) simple  $\Q[G]$-modules.
For a $ \Q[G] $-module $M$,  define
\[M^{\dd}:=M\otimes_{ \Q [G]} \Q [G]^{\dd}.\]
By Wedderburn's Structure Theorem (see \cite[\S 3.5]{Pi}), $ M$ has a decomposition
\[ M=\bigoplus_{ d \sm n} M^{\dd}.\]

Let $ V $ be a $ 2 $-dimension $ \Q $-vector space spanned by $v_{1}$ and $ v_{2} $. Define a pairing on $ V \times V $ by:
\[ \begin{split}
\langle \cdot \ , \ \cdot \rangle : V  \quad  \times  \quad V \qquad &\longrightarrow \Q, \\
\qquad \qquad \qquad ( av_{1}+bv_{2}\ , \, cv_{1}+dv_{2})  &\longmapsto ad-bc
\end{split}   \]
Obviously $ \langle \cdot, \cdot \rangle $ is nondegenerate, and $ \langle v_{1},v_{2}\rangle=1 $. Consider the $ \Q[G] $-module $ \Q[G]\otimes_{\Q}V$, where $ G $ acts on the left. This pairing $ \langle \cdot,\cdot\rangle$ can be naturally extended to 
\begin{equation}\label{pair1}
\begin{split}
\Q[G]\otimes_{\Q} V \quad \times  \quad \Q[G]\otimes_{\Q}V& \longrightarrow \Q, \\
\langle\sum_{g}g\otimes v_{g}\quad , \quad  \sum_{g}g\otimes w_{g}\rangle  &\longmapsto \sum_{g}\langle v_{g},w_{g}\rangle
\end{split} 	
\end{equation}
One can check that this pairing is alternative and nondegenerate, which we still denote as $ \langle\cdot,\cdot\rangle$.

\begin{lem}\label{orth}
For any divisor $ d_{1} $ of $ n $, one has 
\[ 	\big(\Q[G]^{(d_{1})}\otimes_{\Q} V  \big)^{\perp}=\bigoplus_{d_{1} \neq d \sm n}\big(\Q[G]^{\dd}\otimes_{\Q} V\big).  \]
In particular, if $ d_{1} $ and $ d_{2} $ are distinct divisors  of $ n $, then
\[ \langle \Q[G]^{(d_{1})}\otimes_{\Q} V , \Q[G]^{(d_{2})}\otimes_{\Q} V    \rangle =0.  \]
\end{lem}
\begin{proof}
	Suppose $ M $ is a $ \Q[G] $-submodule of $\Q[G]\bigotimes_{\Q} V$, then for any $ a \in M^{\perp} $, $ \langle M,a \rangle=0 $. Thus we have
	\[\langle M,ga\rangle=\langle gM,ga \rangle =\langle M,a \rangle=0,\] 
	which implies that the orthogonal complement of $ \Q[G] $-submodule of $\Q[G]\bigotimes_{\Q} V$ of $ \Q[G]\bigotimes_{\Q} V $ is still a $ \Q[G] $-submodule of $\Q[G]\bigotimes_{\Q} V$.

\vskip 0.3cm	
\noindent	\textbf{Claim}: For any $ 0 \neq v \in V $, $ \Q[G]^{\dd}\otimes_{\Q}\Q v  \nsubseteq  \big( \Q[G]^{\dd}\otimes_{\Q}V\big)^{\perp} $.

\vskip 0.3cm
	If the claim is not true, then there exists $0 \neq v_{0}=av_{1}+bv_{2}\in V  $ such that
	\begin{equation}\label{orth4}
	\langle\Q[G]^{\dd}\otimes_{\Q} V ,\Q[G]^{\dd}\otimes_{\Q} \Q v_{0} \}\rangle=0. 
	\end{equation}
	Without loss of generality, we may assume $ a=1 $. 
	Take $0 \neq T(X) =\sum_{k=0}^{n-1}a_{k}X^{k} \in \Q[X]/\Phi_{d}(X) $. 
	On one hand, by \eqref{pair1}, we have
	\[ 
	\langle T(X)\otimes (v_{1}+bv_{2}),T(X)\otimes v_{2}\rangle  =\sum_{k \geq 0}^{n-1}a_{k}a_{k}>0.
	\]
On the other hand, by \eqref{orth4}, we have
	\[
	\langle T(X)\otimes (v_{1}+bv_{2}),T(X)\otimes v_{2}\rangle=0.
\]
This is a contradiction. Thus follows the claim.
	
	As $ \big(\Q[G]^{(d_{1})}\otimes_{\Q} V  \big)^{\perp} $ is a direct sum of certain simple $ \Q[G] $-modules, the  claim implies that none of simple submodules of $\Q[G]^{(d_{1})}\otimes_{\Q} V $ can appear in the decomposition of $  \big(\Q[G]^{(d_{1})}\otimes_{\Q} V  \big)^{\perp} $, thus
	\begin{equation}\label{subset}
	\big(\Q[G]^{(d_{1})}\otimes_{\Q} V  \big)^{\perp} \subseteq  \bigoplus_{d_{1} \neq d \sm n}\big(\Q[G]^{\dd}\otimes_{\Q} V\big).
	\end{equation}
By counting the dimensions of both sides as $\Q$-vector spaces,  we have
	\[ 	\big(\Q[G]^{(d_{1})}\otimes_{\Q} V  \big)^{\perp} = \bigoplus_{d_{1} \neq d \sm n}\big(\Q[G]^{\dd}\otimes_{\Q} V\big). \qedhere \]
\end{proof}

\begin{proof}[Proof of Theorm~1]
Note that for any divisor $ d $ of $ n $, the  elements 
 \[ \Psi_{n,d}(X), X\Psi_{n,d}(X), \dots,X^{\varphi(d)-1}\Psi_{n,d}(X) \]
form a $ \Q $-basis of $ \Q[X]/\Phi_{d}(X) $. Then we have
\[\begin{split}
X^{l_{1}} \Psi_{n,d_{1}}(X)\otimes v_{1}\in \Q[G]^{(d_{1})}\otimes_{\Q} V, \\
X^{l_{2}} \Psi_{n,d_{2}}(X)\otimes v_{2}\in \Q[G]^{(d_{2})}\otimes_{\Q} V.
\end{split}   \]
 Thus by Lemma~3, we get
\[   \langle X^{l_{1}} \Psi_{n,d_{1}}(X), X^{l_{2}} \Psi_{n,d_{2}}(X)\rangle = \langle X^{l_{1}} \Psi_{n,d_{1}}(X)\otimes v_{1}, X^{l_{2}} \Psi_{n,d_{2}}(X)\otimes v_{2}\rangle=0.  \qedhere  \]
\end{proof}

\end{document}